\documentclass[11pt]{article}
\usepackage{amssymb,amsmath}
\usepackage{amsthm}
\usepackage{amscd}
\usepackage{stmaryrd}
\newtheorem{theorem}{Theorem}[section]
\newtheorem{theorem*}{Theorem A\!\!}
\newtheorem{proposition}{Proposition}[section]
\newtheorem{proposition*}{Proposition A\!\!}

\newtheorem{corollary*}{Corollary A\!\!}

\newtheorem{lemma}{Lemma}[section]

\newtheorem{definition}{Definition}[section]

\DeclareMathOperator{\id}{id}
\DeclareMathOperator{\Id}{Id}
\DeclareMathOperator{\jac}{jac}
\DeclareMathOperator{\tr}{tr}

\DeclareMathOperator{\Det}{Det}

\DeclareMathOperator{\res}{res}

\begin{document}
\title{Rankin-Cohen brackets on tube-type domains}

\author{Jean-Louis Clerc}

\date{March 22, 2020}
\maketitle
\abstract{A new formula is obtained for the holomorphic bi-differential operators on tube-type domains which are associated to the decomposition of the tensor product of two scalar holomorphic representations, thus generalizing the classical \emph{Rankin-Cohen brackets}. The formula involves a family of polynomials of several variables which may be considered as a (weak) generalization of the classical \emph{Jacobi polynomials}.}
\bigskip

{2020 MSC. Primary 22E46; Secondary 32M15, 33C45}

\section*{Introduction}

The tensor product of two holomorphic discrete series representations has been long studied (see e.g. \cite{r}). For scalar ones, L. Peng and G. Zhang (see \cite{pz}), inspired by the work of J. Peetre (see \cite{p}) obtained fairly complete results for the decomposition into irreducible components. When the representations are realized in the weighted Bergman space model, the projectors on the irreducible components are given by bilinear holomorphic differential operators. In the case of the holomorphic discrete series representations of the groupe $SL(2,\mathbb R)$ (or rather of its universal covering), these differential operators coincide with the classical \emph{Rankin-Cohen brackets}. Recently, T. Kobayashi and M. Pevzner (see \cite{kp}) studied the same operators from a different point of view and pointed a link to classical \emph{Jacobi polynomials}. The use of the $L^2$-model for the representations, obtained from the weighted Bergman space model by an inverse Laplace transform plays an important r\^ole in their approach. Their paper was a source of inspiration for this work.

The classical Rankin-Cohen operators are constant coefficients holomorphic bi-differential operators on the upper half-plane, which satisfy a  covariant property withe respect to the group $SL(2,\mathbb R)$. They may equivalently be regarded as constant coefficients bi-differential operators on the real line, acting from $C^\infty(\mathbb R \times \mathbb R)$ into $C^\infty(\mathbb R)$ and satisfying a covariance property with respect to the group $SL(2,\mathbb R)$ acting on the densities on $\mathbb R\times \mathbb R$ (principal series representations). In a previous work, in collaboration with 
S. Ben Sa\"id and K. Koufany (cf \cite{bck}), we constructed a family of covariant bi-differential operators on $V\times V$, where $V$ is a real simple Jordan algebra, covariant under the conformal group of $V$. Elaborating on this result, I obtained later in \cite{c} a closed formula for the symbols of these operators. The symbols are  polynomials (of several variables) which are obtained through a \emph{Rodrigues formula}, similar to the well-known one for the  Jacobi polynomials. 

The present paper comes back to the classical Rankin-Cohen brackets and studies the same problem, replacing the upper half-plane by a (Hermitian symmetric) tube-type domain. Such a space is associated to a Euclidean simple Jordan algebra and the results of \cite{c} can be used to construct  holomorphic bi-differential operators which are covariant with respect to the automorphisms group of the tube-type domain and to the series of holomorphic representations. The covariance property could be obtained by using the results of \cite{c}, but we prefer to offer a new proof written in the spirit of the harmonic analysis on these domains (see \cite{fk} for a general presentation), which is almost self-contained and much shorter. As a by-product, the family of polynomials obtained in \cite{c} is shown to correspond to a family of \emph{orthogonal} polynomials, at least for some values of the parameters (as in the case of the Jacobi polynomials), which correspond to cases where the representations are unitary (holomorphic discrete series). Finally, we suggest some further investigations, which would lead in particular to families of orthogonal polynomials (of several variables) generalizing the Jacobi polynomials.
\bigskip

\centerline{ \bf CONTENTS}
\bigskip

{\bf 1. Geometric and analytic background}
\smallskip

\hskip 1cm 1.1 Euclidean Jordan algebra
\smallskip

\hskip 1cm 1.2 The  automorphisms group of the tube-type domain and

\hskip 1.8cm  the series of holomorphic representations
\smallskip

\hskip 1cm 1.3 Weighted Bergman spaces and the holomorphic discrete

\hskip 1.8cm series

\smallskip

{\bf 2. Rankin-Cohen brackets in a tube-type domain}
\smallskip

\hskip 1cm 2.1 Definition of the Rankin-Cohen brackets
\smallskip

\hskip 1cm 2.2 Expression in the $L^2$-model
\smallskip

\hskip 1cm 2.3 Unitarity and continuity properties
\smallskip

{\bf 3. The covariance property}
\smallskip

\hskip 1cm 3.1 Reproducing kernel and coherent states
 \smallskip
 
\hskip 1cm 3.2 Covariance of the Rankin-Cohen brackets
 
 \smallskip
 
{\bf 4. The family $\big(C^{(k)}_{\lambda, \mu}\big)_{k\in \mathbb N}$}
\smallskip

\hskip 1cm 4.1 The orthogonality property
\smallskip

\hskip 1cm 4.2 Further investigations

\section{Geometric and analytic   background}

\subsection{Euclidean Jordan algebra}

Let $V$ be a simple Euclidean Jordan algebra. A general reference for notation and main results is \cite{fk}. Let $n$ be the dimension of $V$, $r$ its rank and $d$ its characteristic number, which satisfy
\[n= r+\frac{r(r-1)}{2}d\ .
\]
Let $\tr$ and $\det$ be the trace and the determinant of $V$. The inner product on $V$ is given by
\[(x,y) = \tr(xy)\ .
\]
The neutral element is denoted by $e$ and satisfies $(e,x) = \tr x$ for all $x\in V$.

The \emph{structure group}  $Str(V)$ is defined as the subgroup of $GL(V)$ of elements $\ell\in GL(V)$ for which there exists a scalar $\chi(\ell)\in \mathbb R^\times$ such that for all $x\in V$
\begin{equation}\label{chi}
\det(\ell x) = \chi(\ell) \det x\ .
\end{equation}
Then $\chi  : Str(V)\longrightarrow \mathbb R^\times$ is a character of $Str(V)$.

Let $\Omega$ be the open cone of squares, and let $G(\Omega)$ be the subgroup of $GL(V)$ preserving the cone $\Omega$. Then $G(\Omega)\subset Str(V)$ and in fact $Str(V)$ is the direct product of $G(\Omega)$ by $\{ \pm\Id\}$. In particular $G(\Omega)$ and $Str(V)$ have the same neutral component, denoted by $L$. 

For $\ell\in G(\Omega)$, $\chi(\ell)>0$ and
\begin{equation}\label{Detchi}
\Det(\ell) = \chi(\ell)^{n/r}\ .
\end{equation}

Let $P$ be the \emph{quadratic representation}. The following identity holds for $x$ and $y\in V$ :
\begin{equation}\label{chiP}
\det \left(P(x) y \right)= (\det x)^2 \det y\ .
\end{equation}

 Recall the following equivalent propositions for an element $x\in V$ :

\quad $i)$ $x$ is invertible

\quad $ii)$ $\det x\neq 0$

\quad $iii)$ $P(x)$ is invertible\ .

When any of these propositions is satisfied, then the inverse of $x$ is given by
\[x^{-1} = P(x)^{-1} (x)\ .
\]
The open set of invertible elements is denoted by $V^\times$. For $x\in V^\times$, $P(x)$ belongs to $Str(V)$. If moreover $x\in \Omega$, $P(x)$ belongs to $G(\Omega)$ (and even to $L$) and satisfies
 \begin{equation}\label{chiP}
\chi(P(x)) = (\det x)^2\ .
\end{equation}
The \emph{Peirce decomposition} plays the r\^ole of the spectral theorem for the symmetric matrices. In particular it allows to define for any $z\in \Omega$ a square root $z^{1/2}$ which is the unique element in $\Omega$ such that $(z^{1/2})^2=z$.

The following example may help the reader not familiar with Jordan algebras. Let $r$ be an integer, $r\geq 1$, and consider the space $V=Sym(r,\mathbb R)$ of $r\times r$ symmetric matrices with real entries, equipped with the Jordan product $x.y$ given by
\[x.y =\frac{1}{2} (xy+yx)\ .
\]
The rank is $r$, the characteristic number is $d=1$ and the dimension is 
$n=\frac{r(r+1)}{2}\,$. The trace and determinant coincide with the usual notions, the cone $\Omega$ is the cone of positive-definite symmetric matrices. The quadratic representation is given by
\[P(x) = xyx\ .
\]
The group $G(\Omega)$ is isomorphic to $GL(r, \mathbb R)/\{\pm \Id\}$ acting by
\[(g,x) \longmapsto gxg^t\ .
\]

Going back to the general situation, a parametrization of $\Omega\times \Omega$, akin to polar coordinates, will be needed later. Denote by $\rrbracket -e,+e\, \llbracket$ the "interval" between $-e$ and $+e$,  i.e. 
\[\rrbracket-e,+e\,\llbracket = \big(-e+\Omega\big) \cap \big(e-\Omega\big)= \{ x\in V,  e\pm x\in \Omega\}\ .
\]
\begin{proposition}\label{coord}

 The map $\boldsymbol \iota$ 
\begin{equation}
\Omega \times \rrbracket-e,+e\,\llbracket\, \ni (z,v) \longmapsto \boldsymbol \iota (z,v) = \Big( \frac{1}{2} \big(z-P(z^{\frac{1}{2}})v\big), \frac{1}{2} \big(z+P(z^{\frac{1}{2}})v\big)\Big)
\end{equation}
is a diffeomorphism from $\Omega \times \rrbracket-e,+e\,\llbracket $ onto $\Omega\times \Omega$.
Moreover the Jacobian of $\boldsymbol \iota$ is given by
\begin{equation}
\jac(\boldsymbol \iota)(z,v) = 2^{-n} (\det z)^{\frac{n}{r}}\ .
\end{equation}
\end{proposition}

\begin{proof}
Let $z\in \Omega$ and $v\in \rrbracket-e,+e\llbracket$. Then, as $P(z^{1/2})$ belongs to $G(\Omega)$
\[z\pm P(z^{\frac{1}{2}})v = P(z^{\frac{1}{2}})(e\pm v)\in \Omega\ ,
\]
and hence the image of $\boldsymbol \iota$ is contained in $\Omega$.
\smallskip

Conversely, assume $x,y\in \Omega$. Necessarily, $z=x+y$ and as $z\in \Omega$,  $z^{\frac{1}{2}}$ is well-defined. Next, $y-x=P(z^{1/2})v$, and as $P(z^{1/2})$ is invertible, again necessarily $v=P(z^{-\frac{1}{2}})(y-x)$. Now, as $e=P(z^{-\frac{1}{2}})(z)$
\[e\pm v = P(z^{-\frac{1}{2}})(x+y\pm (y-x))\in P(z^{-\frac{1}{2}})(\Omega )= \Omega \ .
\]
Hence $(z,v) \in \Omega \times \rrbracket-e,+e\,\llbracket$. Moreover, 
\[\frac{1}{2}(z\pm P(z^{\frac{1}{2}})v) = \frac{1}{2}(x+y\pm  \big(P(z^{\frac{1}{2}})P(z^{-\frac{1}{2}})\big)(y-x) = \frac{1}{2}\big(x+y\pm (y-x)\big),
\]
or equivalently $\boldsymbol \iota(z,v) = (x,y)$,
 showing that  $\boldsymbol \iota$ is both injective and surjective.
 
 The map $z\longmapsto z^{\frac{1}{2}}$ is known to be a diffeomorphism of $\Omega$, so  that $\boldsymbol \iota$ is a diffeomorphism. To compute the Jacobian of $\boldsymbol \iota$, set for a while $s=x+y$ and $d=-x+y$, so that
 \[s= z, \qquad d= P(z^{\frac{1}{2}}) v,
 \]
 and the differential of the map $(z,v) \longmapsto (s,d)$ is equal to
 \[
 \begin{pmatrix}
 &I&0\\&\bigstar&P(z^{\frac{1}{2}})
 \end{pmatrix}\ .
 \]
 Now by \eqref{chi}, \eqref{Detchi} and \eqref{chiP}, $\Det P(z^{\frac{1}{2}}) =\big( \det (z^{\frac{1}{2}})\big)^{\frac{2n}{r}} =( \det z)^{\frac{n}{r}}$ and the result follows.
\end{proof}
There is a corresponding integration formula for the change of variables.

\begin{proposition}
Let $f$ be an integrable function on $\Omega\times \Omega$. Then
\begin{equation}\label{varchange}
\iint_{\Omega\times \Omega} f(x,y)\, dx dy = 2^{-n} \int_\Omega \int_{\rrbracket-e,+e\llbracket} f(\boldsymbol \iota(z,v)) (\det z)^{\frac{n}{r}} dz dv\ .
\end{equation}
\end{proposition}

\subsection{The  automorphisms group of the tube-type domain and the series of holomorphic representations}

 Form the tube domain $T_\Omega=V\oplus i \Omega\subset \mathbb V = V\otimes \mathbb C$. Let $G(T_\Omega)$ be the group of holomorphic automorphisms of $T_\Omega$. Some subgroups of $T_\Omega$ are easy to describe. Firstly, any $g\in G(\Omega)$ can be extended to a complex linear automorphism of $T_\Omega$ and hence $G(\Omega)$ can be regarded a a subgroup of $G(T_\Omega)$. Next, for $u\in V$, the translation
\[t_u : z\longmapsto z+u
\]
is a holomorphic automorphism of $T_\Omega$, and the group $N= \{t_u, u\in V\}$ is an Abelian subgroup of $G(T_\Omega)$. Finally, any element of $T_\Omega$ is invertible in $\mathbb V$ and the map
\[j: z\longmapsto -z^{-1}
\]
is an holomorphic automorphism of $T_\Omega$.
\begin{proposition}
The subgroups $G(\Omega)$, $N$ together with $j$ generate the group $G(T_\Omega)$.
\end{proposition}
For a proof see \cite{fk} Theorem X.5.6.

Let $G=G(T_\Omega)^0$ be the connected component of $G(T_\Omega)$ containing the neutral element. Notice that $L\subset G$ and also $j$ belongs to $G$, as $j$ belongs to the connected subgroup of $G(T_\Omega)$ given by 
\[\big\{ g_{\alpha, \beta, \gamma,\delta},\quad \alpha, \beta, \gamma,\delta \in \mathbb R,\ \alpha\delta-\beta \gamma=1\big\}/\{\pm \id\}\] where
\[g_{\alpha, \beta,\gamma, \delta} : z\longmapsto  (\alpha z+\beta e)(\gamma z+\delta e)^{-1}\ .
\]
See \cite{fk} p. 208. 

For $g\in G$ and $z\in T_\Omega$ let
\[j(g,z) = \Det_{\,\mathbb C} \big(Dg(z)\big)\ .
\]
For fixed $g$, the function  $z\mapsto j(g,z)$ does not vanish on $T_\Omega$. As $T_\Omega$ is simply connected, it is possible to define a logarithm, and two determinations differ by a multiple of $2i\pi$. 

Now the \emph{universal covering} $\widetilde G$ of $G$ can be described  as follows :
\[\widetilde G = \{\widetilde g =  (g,\psi_g), g\in G, \psi_g : T_\Omega \longrightarrow \mathbb C, e^{\psi_g(z)} = j(g,z) \}
\]
with group law given by
\[(g_1,\psi_{g_1})(g_2,\psi_{g_2}) = \Big(g_1g_2, \psi_{g_1}\big(g_2(z)\big) +\psi_{g_2}(z)\Big)\ .
\]
Notice that the inverse of $(g,\psi_g)$ is given by 
\[(g,\psi_g)^{-1} = (g^{-1}, \psi_{g^{-1}}), \text{ where } \psi_{g^{-1}} =-\psi_g\circ g^{-1}\ .
\]
We omit the description of the topology and the Lie group structure of $\widetilde G$.

There are again subgroups of $\widetilde G$ which are easy to describe, namely those corresponding to the subgroups $N$ and $L$ of $G$ for which we use the same notation :
\[ N= \{ \widetilde t_v = (t_v,0), v\in V\},\qquad  L= \{\widetilde \ell = (\ell, \frac{n}{r}\ln \chi(\ell)), \ \ell \in L\}\ .
\]

 Denote by $\mathcal O(T_\Omega)$ the space of holomorphic functions on $T_\Omega$ equipped with the Montel topology. Let $\nu\in \mathbb C$. Let $\widetilde g=(g,\psi_g)\in \widetilde G$ and let $F\in \mathcal O(T_\Omega)$. The formula 
\begin{equation}\label{defpi}
\pi_\nu(\widetilde g)F(z)= e^{\frac{r}{2n}\nu \psi_{g^{-1}}(z)}F\big(g^{-1}(z)\big) \end{equation}
defines a (smooth) representation $\pi_\nu$ of $\widetilde G$ on $\mathcal O(T_\Omega)$, usually called  the \emph{holomorphic series of representations}.

\subsection{Weighted Bergman spaces and the holomorphic discrete series}

Let $dx$ be the Lebesgue measure on $V$ associated to the Euclidean structure on $V$. Recall the \emph{Gamma function} of the positive cone $\Omega$ given by
\begin{equation}\label{Gamma}
\Gamma_\Omega(\nu) = \int_\Omega e^{-\tr(x)} \det(x)^{\nu-\frac{n}{r}} dx\ .
\end{equation}
The integral converges absolutely for $\Re(\nu) > (r-1) \frac{d}{2}$, has a meromorphic continuation to $\mathbb C$ and is given by
\[ \Gamma_\Omega(\nu) = (2\pi)^{\frac{n-r}{2}}\Gamma(\nu)\Gamma\left(\nu-\frac{d}{2}\right)\dots \Gamma\left(\nu-(r-1)\frac{d}{2}\right) \ .
\]
See \cite{fk} chapter VII.

For a real parameter $\nu$, define $\mathcal H_\nu(T_\Omega)$ as the space of holomorphic functions $f: T_\Omega \longrightarrow \mathbb C$ such that
\[\Vert F\Vert_\nu^2 = \int_{T_\Omega} \vert F(x+iy)\vert^2 \det(y)^{\nu-\frac{2n}{r}} \,dx\,dy\ <\ +\infty\ .
\]
For $\nu\leq 1+d(r-1)$, the space $\mathcal H_\nu$ is reduced to $\{0\}$, so  assume that $\nu > 1+d(r-1)$. Then $\mathcal H_\nu$ is a Hilbert space $\neq \{0\}$. For these results, see \cite{fk}, chapter XIII.

For $(g, \psi_g)\in \widetilde G$ and $F\in \mathcal H_\nu$ the formula
\begin{equation}
\pi_\nu(\widetilde g) F(z) = e^{\frac{r\nu}{2n}\psi_{g^{-1}}(z)}\, F(g^{-1}(z))
\end{equation}
defines now a \emph{unitary} representation $\pi_\nu$ on $\mathcal H_\nu$. These representations belong to  the \emph{holomorphic discrete series}. 

There is another usueful realization of these representations, using the \emph{Laplace transform}. For $f\in C^\infty_c(\Omega)$, let for $z\in T_\Omega$

\[\mathcal Lf(z) = \int_\Omega f(\xi) e^{i(z,\xi)} d\xi\ .
\]
The integral converges for $z\in T_\Omega$ and  the function $\mathcal Lf$ is holomorphic on $T_\Omega$.

Let $L^2(\Omega)_\nu= L^2\left(\Omega, \det(\xi)^{-\nu+\frac{n}{r}}\right)$ be the Hilbert space of functions on $\Omega$ which are square-integrable w.r.t. the measure $\det(\xi)^{-\nu+\frac{n}{r}} d\xi$.
\begin{proposition} For $\nu \in \mathbb R, \nu>1+d(r-1)$, the Laplace transform can be extended by continuity to yield (up to a scalar) an isometry from $L^2(\Omega)_\nu$ onto $\mathcal H_\nu(T_\Omega)$.
Furthermore
\[\Vert \mathcal L f\Vert^2_\nu =2^{\frac{n}{r}-\nu} \Gamma_\Omega(\nu-\frac{n}{r})\Vert f\Vert^2_\nu \ .
\]
\end{proposition}
For this result, see again \cite{fk} ch. XIII. Notice however some changes in notation.

\section{Rankin-Cohen brackets in a tube-type domain}

\subsection{Definition of the Rankin-Cohen brackets}

We now proceed to the construction of a family of holomorphic bi-differential operators with constant coefficients on $\mathbb V\times \mathbb V$, called the \emph{(generalized) Rankin-Cohen brackets}.

\begin{proposition} Let $s,t\in \mathbb C$ and $k\in \mathbb N$. The expression 
\begin{equation}
 c_{s,t}^{(k)}(x,y) :=(\det x)^{-s} (\det y)^{-t}\det\left(\frac{\partial }{\partial x} -  \frac{\partial }{\partial y}\right)^k\left((\det x)^{s+k} (\det y)^{t+k}\right)
\end{equation}
a priori defined on $\Omega \times \Omega$ can be extended as a polynomial function on $V\times V$. Moreover the coefficients of  the polynomial depend polynomially on the parameters $s,t$.
\end{proposition}
This proposition, formulated in  \cite{c} is essentially a by-product of an important result of \cite{bck}.
\begin{lemma} The polynomial $c_{s,t}^{(k)}$ satisfies the identity
\begin{equation}\label{cst}
\forall \ell\in G(\Omega),\qquad \qquad  c_{s,t}^{(k)}(\ell x,\ell y)= \chi(\ell)^k\, c_{s,t}^{(k)}(x,y)\ .
\end{equation}
\end{lemma}
\begin{proof} Let  $D= \det\left(\frac{\partial }{\partial x} -  \frac{\partial }{\partial y}\right)$. For $\ell\in G(\Omega)$, let $D_\ell$ be the differential operator given by
\[D_\ell f = D(f\circ \ell^{-1})\circ \ell\
\]
The symbol of $D$ is  $d(\xi, \zeta)=i^r \det(\xi-\zeta)$, and the symbol of $D_\ell$ is equal to $d_\ell=d\circ (\ell^t)^{-1}$. From \eqref{chi} follows for $h\in G(\Omega)$
\[d(h\xi,h\zeta) =\chi(h)\, d(\xi,\zeta)\ .
\]
Apply to $h=(\ell^t)^{-1}$ to obtain
\[d_\ell= \chi(\ell)^{-1} d \qquad \text{ and hence }\qquad  D_l = \chi(\ell)^{-1} D\ ,
\]
which implies
\begin{equation}\label{Dk}
D_\ell^k = \chi(\ell)^{-k} D^k\ .
\end{equation}
Let  for a while $e_{s,t}(x,y)=\det(x)^{s}\det(y)^{t}$. By definition,
\[D^k e_{s+k,t+k} = c_{s,t}\, e_{s,t}\ .
\]
On the other hand,
\[D_\ell^k e_{s+k,t+k}=D^k (e_{s+k,t+k} \circ \ell^{-1}) \circ \ell= (\chi(\ell)^{-s-t-2k} (D^k e_{s+k,t+k})\circ \ell
\]
\[=\chi(\ell)^{-s-t-2k} (c_{s,t}\, e_{s,t})\circ \ell= \chi(\ell)^{-2k} (c_{s,t} \circ \ell)\, e_{s,t}\ .
\]
Combining these two calculations with \eqref{Dk} yields \eqref{cst}. Notice that this result implies that $c_{s,t}^{(k)}$ is homogeneous of degree $rk$.
\end{proof}

Let $\res : \mathcal O(T_\Omega\times T_\Omega)\longrightarrow \mathcal O (T_\Omega)$ be the \emph{restriction map} to the diagonal given by
\[f\in  \mathcal O(T_\Omega\times T_\Omega),\qquad \res(f)(z) = f(z,z)\ .
\]
\begin{definition}
For $\lambda, \mu \in \mathbb C$ and $k\in \mathbb N$ define the \emph{$k$-th generalized Rankin-Cohen bracket} $B_{\lambda, \mu}^{(k)}$ by
\begin{equation}B_{\lambda, \mu}^{(k)} = \res\, \circ \, c_{\lambda-\frac{n}{r},\, \mu-\frac{r}{n}} \left(\frac{\partial}{\partial z}, \frac{\partial}{\partial w}\right)\ .
\end{equation}
\end{definition}
\subsection{Expression in the $L^2$-model}

Assume now that $\lambda,\mu>1+d(r-1)$. The (completed) tensor product  $\mathcal H_\lambda\widehat \otimes \mathcal H_\mu$ is identified with the space $\mathcal H_{\lambda, \mu}$ of holomorphic functions  $F$  on $T_\Omega \times T_\Omega$ such that
\[\iint_{T_\Omega \times T_\Omega} \vert F(x+iy,u+it)\vert^2 \det(y)^{\lambda-\frac{2n}{r}} \det(t) ^{\mu-\frac{2n}{r}} dx\, dy\, du\, dt<\infty\ .
\]
The corresponding (completed) tensor product $L_\lambda^2(\Omega)\widehat \otimes L_\mu^2(\Omega)$ is identified with \[L^2_{\lambda,\mu}(\Omega\times \Omega) = L^2\big(\Omega\times \Omega, \det(\xi)^{-\lambda+\frac{n}{r}} \det(\zeta)^{-\mu+\frac{n}{r}}\,d\xi\, d\zeta\big)\ .\]
For $f\in C^\infty_c(\Omega \times \Omega)$ define its Laplace transform $\mathcal L_2f$ to be 
\[\mathcal L_2f(z,w) = \iint_{\Omega\times \Omega}f(\xi,\zeta) e^{i((z,\xi)+(w,\zeta))} d\xi d\zeta\ .
\]
Then $\mathcal L_2$ extends as an isometry (up to a scalar) between $L^2_{\lambda, \mu}(\Omega\times \Omega)$ and $\mathcal H_{\lambda,\mu}$.

The Rankin-Cohen operators have a counterpart when working with the $L^2$-model instead of the weighted Bergman spaces. Define
\[\widehat B_{\lambda, \mu}^{(k)} = \mathcal L^{-1} \circ B_{\lambda, \mu}^{(k)} \circ \mathcal L_2 =\mathcal L^{-1}\, \circ \res\circ \,c_{\lambda-\frac{n}{r},\, \mu-\frac{n}{r}}\left(\frac{\partial}{\partial z},  \frac{\partial}{\partial w}\right)\,\circ \mathcal L_2\ .
\]
The composition of these operators is {\it a priori } formal, but will be shown to make sense  on functions in $C^\infty_c(\Omega\times \Omega)$. Recall the elementary formula for the Laplace transform, valid for any holomorphic polynomial $p$ on $\mathbb V$
\[p\left(\frac{\partial}{\partial z} \right)(\mathcal Lf)(z) = \int_\Omega p(i\xi) f(\xi) e^{i(z,\xi)}\  d\xi
\]
The following consequence is then immediate.
\begin{lemma}
\begin{equation}
c_{\lambda-\frac{n}{r},\, \mu-\frac{n}{r}}\left(\frac{\partial}{\partial z},  \frac{\partial}{\partial w}\right)\circ \mathcal L_2 = i ^{rk}\,\mathcal L_2\circ c_{\lambda-\frac{n}{r},\, \mu-{\frac{n}{r}}}\ .
\end{equation}
\end{lemma}

Notice that the multiplication by the polynomial $c_{\lambda-\frac{n}{r},\, \mu-{\frac{n}{r}}}$ is a continuous operator on $C^\infty_c(\Omega\times \Omega)$.

The next step is to calculate $\res\circ \mathcal L_2$ (cf \cite{kp}). For $f\in C^\infty_c(\Omega\times \Omega)$,  define for $\xi\in \Omega$ 
\begin{equation}
\mathcal J f(\xi)=2^{-n}(\det \xi)^{\frac{n}{r}}\int_{-\rrbracket e,e\llbracket} f(\boldsymbol \iota(\xi,v)) dv\ .
\end{equation}
By elementary arguments, the integral converges, the resulting function $\mathcal J f$ belongs to $C^\infty_c(\Omega)$ and the operator $\mathcal J: C^\infty_c(\Omega\times \Omega)\longrightarrow C^\infty_c(\Omega)$ is continuous. 
\begin{lemma}\begin{equation}\label{Laplaceres}
\res \circ \,\mathcal L_2 = \mathcal L\circ \mathcal J \ .
\end{equation}
\end{lemma}
\begin{proof}
 Let $f\in C^\infty_c(\Omega\times \Omega)$. Then, for $z\in T_\Omega$

\[\res \big(\mathcal L_2 f\big) (z) = \int_\Omega \int_\Omega f(\xi,\zeta) e^{i(z,\xi+\zeta)} d\xi d\zeta
\]
\[ = \int_\Omega \left(2^{-n} (\det \eta)^{\frac{n}{r}} \int_{\rrbracket-e,+e\,\llbracket} f\big(\boldsymbol \iota(\eta,v)\big)dv \right)e^{i(z,\eta)} d\xi\ ,
\]
by using the change of variables $(\xi,\zeta) = \boldsymbol \iota(\eta,v)$ and the integration formula \eqref{varchange}. This finishes the proof.
\end{proof}
From these two lemmas follows the main result.
\begin{theorem}\label{hatB}
 The operator $\widehat B_{\lambda,\mu}^{(k)}$ maps $ C^\infty_c(\Omega\times \Omega)$ into $C^\infty_c(\Omega)$ and satisfies
\begin{equation}
\widehat B_{\lambda,\mu}^{(k)}=i^{rk} \, \mathcal J\circ c_{\lambda- \frac{n}{r},\,\mu- \frac{n}{r}}^{(k)}\ .
\end{equation}
\end{theorem}
The expression for the operator $\widehat B_{\lambda,\mu}^{(k)}$ can be given a slightly different form. Introduce the polynomial $C^{(k)}_{\lambda, \mu}$ on $V$ defined by
\[C^{(k)}_{\lambda,\mu}(x) = c^{(k)}_{\lambda,\mu}\left(\frac{e-x}{2},\frac{e+x}{2}\right)\ .
\]
\begin{lemma} For $(\eta,v)\in \Omega \times \rrbracket -e, +e\,\llbracket$
\begin{equation}
c_{\lambda,\mu}^{(k)}(\boldsymbol \iota(\eta,v)) = (\det \eta)^k C_{\lambda, \mu}^{(k)}(v)\ .
\end{equation}
\end{lemma}
\begin{proof} Use the covariance property \eqref{cst} of the polynomials $c_{\lambda,\mu}^{(k)}$ under the action of $G(\Omega)$ to obtain 
\[c_{\lambda,\mu}^{(k)}(\boldsymbol \iota(\eta,v))= c_{\lambda,\mu}^{(k)}\left(P(\eta^{1/2})\frac{e-v}{2}, P(\eta^{1/2})\frac{e+v}{2}\right) 
\]
\[=\chi\left(P(\eta^{1/2})\right)^k c_{\lambda,\mu}^{(k)}\left(\frac{e-v}{2}, \frac{e+v}{2}\right)=  (\det \eta)^kC_{\lambda, \mu}^{(k)} (v)\ ,
\]
by using \eqref{chiP}.
\end{proof}
Hence we may rewrite Theorem \ref{hatB} as follows.
\begin{theorem} 
\begin{equation}\label{hatB2}
\widehat B_{\lambda, \mu}^{(k)} f (\xi) = 2^{-n} i^{rk}(\det \xi)^{k+\frac{n}{r}} 
\int_{\rrbracket -e, e\llbracket}C_{\lambda-\frac{n}{r},\, \mu-\frac{n}{r}}(v) f\left(\boldsymbol \iota(\xi,v)\right) dv\ .
\end{equation}
\end{theorem}
\subsection{Continuity and unitarity properties}

When $\lambda, \mu$ are real numbers and $\lambda, \mu > 1+d(r-1)$, it is possible to study continuity and unitarity properties of the Rankin-Cohen brackets using their simple expression in the $L^2$-model. 

Let $h\in C^\infty_c(\Omega)$. The formula
\begin{equation}
\begin{split}
\left(\Phi_{\lambda, \mu}^{(k)} h\right) (\xi,\zeta) =& i^{rk}(\det \xi)^{\lambda-\frac{n}{r}}(\det \zeta)^{\mu-\frac{n}{r}}\det(\xi+\zeta)^{-\lambda-\mu-2k+\frac{n}{r}}\\ &c_{\lambda-\frac{n}{r},\, \mu-\frac{n}{r}}(\xi,\zeta) h(\xi +\zeta)
\end{split}
\end{equation}
defines an operator from $C_c^\infty(\Omega)$ into $C^\infty(\Omega\times \Omega)$.

The Hilbert product on $L^2_\nu(\Omega)$ (resp. on $L^2_{\lambda,\nu}(\Omega\times \Omega)$) induces a duality between $C^\infty_c(\Omega)$ and $C^\infty(\Omega)$ (resp. $C^\infty_c(\Omega\times \Omega)$ and $C^\infty(\Omega\times \Omega)$), denoted by $(h_1,h_2)_{L^2_\nu}$ (resp. $(f_1,f_2)_{L^2_{\lambda, \mu}}$).

\begin{proposition} The operator $\Phi_{\lambda, \mu}^{(k)}$ is the formal adjoint of $\widehat B_{\lambda, \mu}^{(k)}$ with respect to the dualities induced by the Hilbert products on $L^2_{\lambda+\mu+2k}(\Omega)$ and $L^2_{\lambda,\mu}(\Omega\times \Omega)$.
\end{proposition}

\begin{proof} Let $h\in C^\infty_c(\Omega)$ and $f \in C^\infty_c(\Omega\times \Omega)$. Then
\[ \left(h, \widehat B_{\lambda, \mu}^{(k)}f\right)_{L^2_{\lambda+\mu+2k}}= 
\]
\[2^{-n}i^{rk}\int_\Omega h(\xi) (\det \xi)^{k+\frac{n}{r}}\det(\xi)^{-\lambda-\mu-2k+\frac{n}{r}} \int_{\rrbracket -e, +e\, \llbracket} C_{\lambda-\frac{n}{r},\, \mu-\frac{n}{r}}^{(k)}(v)\overline{f(\boldsymbol \iota(\xi,v))} dv d\xi\ .
\]
On the other hand,
\[\left( \Phi_{\lambda,\mu}^{(k)} h, f\right)_{L^2_{\lambda,\mu}}=
\]
\[i^{rk}\iint_{\Omega\times \Omega} h(\xi+\zeta)c_{\lambda-\frac{n}{r},\,\mu-\frac{n}{r}}(\xi,\zeta)\det(\xi+\zeta)^{-\lambda-\mu-2k+\frac{n}{r}}\overline{f(\xi,\zeta)} d\xi d\zeta
\]
\[=2^{-n}i^{rk} \int_\Omega \int_{\rrbracket -e,e\,\llbracket} h(\eta)(\det \eta)^{\frac{n}{r}} (\det \eta)^{-\lambda-\mu-2k+\frac{n}{r}}(\det \eta)^k C_{\lambda-\frac{n}{r},\, \mu-\frac{n}{r}}(v) \overline {f(\boldsymbol \iota(\eta,v)}) dv d\eta
\]
and the result follows.
\end{proof}

\begin{proposition}\label{Phiisom}
 The operator $\Phi_{\lambda,\mu}^{(k)}$ is a partial isometry (up to a scalar) from $L^2_{\lambda+\mu+2k}(\Omega)$ into $L_{\lambda, \mu}^2(\Omega\times\Omega)$.
\end{proposition}
\begin{proof}
\[\Vert \Phi_{\lambda, \mu}^{(k)}h\Vert^2 = 
\]
\[\int_\Omega \int_\Omega \vert \Phi_{\lambda, \mu}^{(k)}h(\xi,\zeta)\vert^2 (\det \xi)^{-\lambda+\frac{n}{r}} (\det \zeta)^{-\mu+\frac{n}{r}} d\xi d\zeta
\]
\[=\int_\Omega \int_\Omega \vert h(\xi+\zeta)\vert^2 \vert c^{(k)}_{\lambda-\frac{n}{r}, \mu-\frac{n}{r}}(\xi,\zeta)\vert^2 (\det \xi)^{\lambda-\frac{n}{r}}(\det \zeta)^{\mu-\frac{n}{r}}\det(\xi+\zeta)^{-2\lambda-2\mu-4k+\frac{2n}{r}}
 d\xi d\zeta\ .
 \]
 Now use the coordinate change $(\xi,\zeta) = \boldsymbol \iota(\eta,v)$. Notice that
 \[\det \xi = 2^{-r} \det \eta \,\det(e-v),\quad \det \zeta = 2^{-r} \det \eta \,\det(e+v),
 \]
 \[ c_{\lambda-\frac{n}{r}, \mu-\frac{n}{r}}^ {(k)}(\xi, \zeta) =(\det \eta)^k\, C^{(k)}_{\lambda-\frac{n}{r},\, \mu-\frac{n}{r}}(v) \ ,
 \]
 so that the integral becomes
 \[2^{-r\lambda-r\mu+n} \int_{ \rrbracket-e,+e\llbracket }\vert C^{(k)}_{\lambda-\frac{n}{r}, \mu-\frac{n}{r}}(v)\vert^2 \det(e-v)^{\lambda-\frac{n}{r}} \det(e+v)^{\mu-\frac{n}{r}} dv
 \]
 \[\times \int_\Omega \vert h(\eta)\vert^2(\det \eta) ^{-\lambda-\mu-2k+\frac{n}{r}} d\eta
 \]
 \[= c(\lambda, \mu; k) \Vert h\Vert^2_{\lambda+\mu+2k},
 \]
 where
 \[c(\lambda, \mu; k) = 2^{-r\lambda-r\mu+n} \int_{ \rrbracket-e,+e\llbracket }\vert C_{\lambda-\frac{n}{r}, \mu-\frac{n}{r}}(v)\vert^2 \det(e-v)^{\lambda-\frac{n}{r}} \det(e+v)^{\mu-\frac{n}{r}} dv\ .
 \]
\end{proof}

\begin{theorem} The operator $B_{\lambda, \mu}^{(k)}$ is continuous from $\mathcal H_{\lambda, \mu}$ into $\mathcal H_{\lambda+\mu+2k}$. Its adjoint ${B_{\lambda, \mu}^{(k)}}^*$ is (up to a scalar) a partial isometry from $\mathcal H_{\lambda+\mu+2k}$ into $\mathcal H_{\lambda, \mu}$.
\end{theorem}

\begin{proof} Recall that $\Phi_{\lambda, \mu}^{(k)}$ is  the dual of the operator $\widehat B_{\lambda, \mu}^{(k)}$. From Proposition \ref{Phiisom} follows that $\widehat B_{\lambda, \mu}^{(k)}$ is continuous from $L^2_{\lambda,\mu}(\Omega\times \Omega)$ into $L^2_{\lambda+\mu+2k}(\Omega)$. Now via Laplace transform/inverse Laplace transform, $B_{\lambda, \mu}^{(k)}$ is continuous from $\mathcal H_{\lambda, \mu}$ into $\mathcal H_{\lambda+\mu+2k}$. Similarly, its adjoint ${B_{\lambda, \mu}^{(k)}}^*$ corresponds to $\Phi_{\lambda,\mu}^{(k)}$ and hence is a partial isometry (up to a scalar).
\end{proof}

Said differently, the last result shows that the orthogonal decomposition of the tensor product $\pi_\lambda\otimes \pi_\mu$ into irreducible representations contains a copy of $\pi_{\lambda+\mu+2k}$ for each $k\in \mathbb N$ (cf \cite{pz}).

\section{The covariance property}

The Rankin-Cohen bracket $B^{(k)}_{\lambda, \mu}$ satisfies a covariance property  under the action of $\widetilde G$. The proof of this property is achieved in three steps. For the first two steps, we assume that $\lambda,\mu>1+(r-1)d$. 

The first step is to calculate the image by the adjoint operator   ${B^{(k)}_{\lambda, \mu}}^*$ of a specific vector  $k_\nu^{ie} \in \mathcal H_{\lambda+\mu+2k}$ . This is done in subsection 3.1.

In the second step (subsection 3.2), the covariance under the group of translations and under the group $L$ is easily obtained in the $L^2$-model of the representations. In turn, this limited covariance property allows to compute the image  by the adjoint operator   ${B^{(k)}_{\lambda, \mu}}^*$ of all \emph{coherent states} $(k_\nu^w)_{w\in T_\Omega}$ from the computation made in the first step for $k_\nu^{ie}$. The covariance property of  ${B^{(k)}_{\lambda, \mu}}^*$ under $\widetilde G$ is then verified on the coherent states and hence globally by continuity. The desired covariance property of $B^{(k)}_{\lambda,\mu}$ under $\widetilde G$ follows. However, the covariance property holds for $B^{(k)}_{\lambda,\mu}$ viewed as an operator from $\mathcal H_{\lambda,\mu}$ into $\mathcal H_{\lambda+\mu+2k}$. An extra argument extends the property to the corresponding Montel spaces.

The third step is an analytic continuation argument which allows to extend the covariance property of $B^{(k)}_{\lambda,\mu}$ to arbitrary parameters $\lambda,\mu\in \mathbb C$.

\subsection{ Reproducing kernel and coherent states}

Let $z\in T_\Omega$. Then $\det z\neq 0$. As $T_\Omega$ is simply connected, there exists a global definition of $\log (\det \left(\frac{z}{i}\right))$ and the natural choice is the one such that $ \log (\det \left(\frac{iy}{i}\right)) = \ln (\det y)$ for $y\in \Omega$. This convention is tacitly used in the sequel, in particular for defining the powers
\[det\left( \frac{z}{i}\right)^\nu= e^{\nu \log (\det  (\frac{z}{i}))}, \qquad z\in T_\Omega\ .
\]
For $\nu>1+(r-1)d$, the reproducing kernel of $\mathcal H_\nu(T_\Omega)$ (after renormalization of the norm)  is given by
\begin{equation}
k_\nu(z,w) = \det\left( \frac{z-\overline w}{i}\right)^{-\nu}\ .
\end{equation}
See \cite{fk} Proposition XIII.1.2.

Let $k_\nu^w$ be the \emph{coherent state} associated to $w\in T_\Omega$ defined by
\[k_\nu^w(z) = k(z,w)=  \det\left( \frac{z-\overline w}{i}\right)^{-\nu}
\]
which belongs to $\mathcal H_\nu$.

First  consider the case $w=ie$. Adapting the notation, define

\[\varphi_\nu (z) = k_\nu^{ie}(z) = \det(\frac{z+ie}{i})^{-\nu}, \qquad z\in T_\Omega \ .
\]
Let also
\[\psi_\nu (\xi) = e^{-\tr \xi} (\det \xi)^{\nu-\frac{n}{r}},\qquad x\in \Omega\ .
\]
\begin{lemma} Assume $\nu>1+(r-1)d$. Then
\smallskip

$i)$ $\psi_\nu$ belongs to $L^2_\nu(\Omega)$
\smallskip

$ii)$ $\mathcal L (\psi_\nu) = \Gamma_\Omega(\nu) \varphi_\nu$\ .

\end{lemma} 

\begin{proof}

For $i)$ 
\begin{equation*}
\begin{split}
\Vert \Psi_\nu\Vert^2=\int_\Omega e^{2 \tr \xi}\det \xi^{\nu-\frac{n}{r}} d\xi
\\
 = 2^{-r\nu}\int_\Omega e^{-\tr \eta} (\det \eta)^{\nu-\frac{n}{r}} d\eta\ .
\end{split}
\end{equation*}
By assumption, $\nu>1+(r-1)d$ and hence $\nu-\frac{n}{r} >\frac{r-1}{2} d$, so that $\nu\frac{n}{r}$ belongs to the domain of absolute convergence of  the integral defining $\Gamma_\Omega$ (see comments after the definition of $\Gamma_\Omega$ \eqref{Gamma}).

Let $y\in \Omega$ and make the change of variable $\xi = P(y^{1/2})\eta$. As
\[(e,P(y^{1/2}) \eta) = (y,\eta)\ , \qquad \det P(y^{1/2}) \eta = \det y \det \eta\ , \qquad d\xi = (\det y)^{\frac{n}{r}} d\eta\ ,
\]
follows
\[\int_\Omega e^{-(y,\eta)}(\det \eta)^{\nu-\frac{n}{r}} d\eta = \Gamma_\Omega(\nu) (\det y)^{-\nu}\ .
\]
This identity can be extended analytically to $y$ in the right half-plane, and in particular to $\displaystyle y=e-iz = \frac{z+ie}{i}$ for $z\in T_\Omega$, thus yielding $ii)$.
\end{proof}

\begin{proposition}\label{imphi}
 Let $\lambda, \mu > 1+(r-1)d$ and let $k\in \mathbb N$. Then
\begin{equation*}
\begin{split}
{B_{\lambda, \mu}^{(k)}}^* \varphi_{\lambda+\mu+2k} &(z_1,z_2)\\= (-i)^{rk} \Gamma_\Omega(\lambda+k)\Gamma(\mu+k)
\det(z_1-z_2)^k &\det\big(\frac{z_1+ie}{i}\big)^{-\lambda-k} \det\big(\frac{z_2+ie}{i}\big)^{-\mu-k}\ .
\end{split}
\end{equation*}
\end{proposition}

\begin{proof}  First, recall from the previous section that $\Phi_{\lambda, \mu}^{(k)}$ corresponds under Laplace/inverse Laplace transforms to ${B_{\lambda, \mu}^{(k)}}^* $. Hence it is wise to first calculate $\Phi_{\lambda,\mu}^{(k)} \psi_{\lambda+\mu+2k}$ and then take its Laplace transform. Now
\begin{equation*}
\Phi_{\lambda,\mu}^{(k)}(\psi_{\lambda+\mu+2k})(\xi_1, \xi_2) =i^{rk} (\det \xi_1)^{\lambda-\frac{n}{r}} (\det \xi_2)^{\mu-\frac{n}{r}} e^{-\tr \xi_1} e^{-\tr \xi_2} \, c_{\lambda-\frac{n}{r},\, \mu-\frac{n}{r}}^{(k)}(\xi_1, \xi_2)\ .
\end{equation*}
The Laplace transform of this function is
\[i^{rk}\int_\Omega \int_\Omega e^{i\big((z_1,\xi_1)+(z_2,\xi_2)\big)} \det \xi_1^{\lambda-\frac{n}{r}} \det \xi_2^{\mu-\frac{n}{r}} e^{-\tr \xi_1} e^{-\tr \xi_2} \, c_{\lambda-\frac{n}{r},\, \mu-\frac{n}{r}}^{(k)}(\xi_1, \xi_2) \, d\xi_1 \, d\xi_2\ .
\]
Recall that 
\[\det \xi_1^{\lambda-\frac{n}{r}} \det \xi_2^{\mu-\frac{n}{r}}c_{\lambda-\frac{n}{r},\, \mu-\frac{n}{r}}^{(k)}(\xi_1, \xi_2) = \det\left(\frac{\partial}{\partial \xi_1}-\frac{\partial}{\partial \xi_2}\right)^k\left( \det \xi_1^{\lambda+k-\frac{n}{r}} \det \xi_2^{\mu+k-\frac{n}{r}} \right),
\]
and substitute this expression in the integral to obtain
\[i^{rk}\int_\Omega \int_\Omega e^{i\big((z_1,\xi_1)+(z_2,\xi_2)\big)} e^{-\tr(\xi_1+\xi_2)}\det\left(\frac{\partial}{\partial \xi_1}-\frac{\partial}{\partial \xi_2}\right)^k\left( \det \xi_1^{\lambda+k-\frac{n}{r}} \det \xi_1^{\mu+k-\frac{n}{r}} \right) d\xi_1 d\xi_2\ .
\]
Next integrate by parts. Notice that for  any smooth function  $f$ on $\Omega$ and any polynomial $p$ on $V$
\[p\left(\left(\frac{ \partial}{\partial \xi_1} - \frac{ \partial}{\partial \xi_2}\right)\right) f(\xi_1+\xi_2) = 0, 
\]
so that 
\[\det\left(\frac{\partial}{\partial \xi_1}-\frac{\partial}{\partial \xi_2}\right)^k \left(e^{i\big((z_1,\xi_1)+(z_2,\xi_2)\big)} e^{-\tr(\xi_1+\xi_2)} \right)\]
\[= i^{rk} \det(z_1-z_2)^k e^{i\big((z_1,\xi_1)+(z_2,\xi_2)\big)} e^{-\tr(\xi_1+\xi_2)}\ .
\]
Moreover, the condition on $\lambda$ and $\mu$ implies $\lambda-\frac{n}{r}, \mu-\frac{n}{r}>(r-1)\frac{d}{2}$, which in turn implies the vanishing of the border contributions when performing the integrations by parts. Thus the integral is equal to
\[(\det(z_1-z_2)^k  \int_\Omega \int_\Omega  e^{-\tr\xi} \det \xi^{\lambda+k-\frac{n}{r}} e^{-\tr \zeta}\det \zeta^{\mu+k-\frac{n}{r}} d\xi d\zeta
\]
\[=\Gamma_\Omega(\lambda+k)\Gamma_\Omega(\mu+k)\det(z_1-z_2)^k \det\big( \frac{z_1+ie}{i}\big)^{-\lambda-k} \det\big(\frac{z_2+ie}{i}\big)^{-\mu-k}\ .
\]
\end{proof}
The next two lemmas will be needed in the sequel.
\begin{lemma} Let $\nu>1+(r-1)d$. For any $\widetilde g=(g,\psi_g)\in \widetilde G$ and $w\in T_\Omega$, the following identity holds true :
\begin{equation}\label{kwkgw}
\pi_\nu(\widetilde g)k_\nu^w = e^{\frac{r\nu}{2n}\, \overline{\psi_g(w)}}\,\, k_\nu^{g(w)}\ .
\end{equation}
\end{lemma}

\begin{proof} Let  $F\in \mathcal H_\nu$. Then
\[\pi_\nu(\widetilde g)F(w)= \langle \pi_\nu(\widetilde g) F, k_\nu^w\rangle = \langle F, \pi_\nu(\widetilde g^{-1})k^w_\nu\rangle\ ,\]
and on the other hand,
\[ \pi_\nu(\widetilde g)F(w)= e^{\frac{r\nu}{2n}\psi_{g^{-1}}(w)} F\big(g^{-1}(w)\big)
=e^{\frac{r\nu}{2n}\psi_{g^{-1}}(w)} \langle F, k_\nu^{g^{-1}(w)}\rangle\ .
\]
Hence
\[\pi_\nu(\widetilde g^{-1}) k_\nu^w = e^{\frac{r\nu}{2n}\overline{\psi_{g^{-1}(w)}}}\,k^{g^{-1}(w)}_\nu\ .
\]
\end{proof}
\begin{lemma} Let $g\in G(T_\Omega)$. Then for any $z,w \in T_\Omega$
\begin{equation}\label{covdet}
\det(g(z)-\overline{g(w)}) = j(g, z)^{1/2} \det(z-\overline w)\, \overline{j(g, w)}^{1/2}\ .
\end{equation}
with the convention that $j(g, z)^{1/2}$ and $j(g, w)^{1/2}$ are computed using the same determination of the square root of $j(g,.)$.
\end{lemma} 
This is a classical result, which can be checked on the generators of the group $G(T_\Omega)$. The equality is easy for $g$ either a translation of an element of $L$ and is a consequence of \emph{Hua's formula} when $g=j$ (see \cite{fk} Lemma X.4.4).

\subsection{Covariance of the Rankin-Cohen brackets}

As explained in the introduction to this section, the covariance property is easy to obtain for the subgroups $N$ and $L$. First recall the action of these groups in the space $\mathcal H_\nu(T_\Omega)$ ($\nu>1+(r-1)d$)

\[v\in V,\qquad \pi_\nu(t_v) F(z) = F(z-v)
\]
\[\ell\in L\qquad \pi_\nu (\ell) F(z) = \chi(\ell)^{-\frac{\mu}{2}}F(l^{-1}z)
\]

Let $\widetilde \pi_\nu$ the representation of $G$ on $L^2_\nu(\Omega)$ corresponding to $\pi_\nu$ transmuted through the Laplace transform. Then
\[\big(\widetilde \pi_\nu(t_v)\varphi\big) (\xi) = e^{-i(v,\xi)} \varphi(\xi), \qquad v\in V
\]
\[\big(\widetilde \pi_\nu(\ell) \varphi\big)(\xi)= \chi(\ell)^{-\frac{\nu}{2}+\frac{n}{r}}\, \varphi(\ell^* \xi), \qquad \ell\in L\ .
\]
Let as before $\lambda, \mu>1+(r-1)d$ and let $k\in \mathbb N$.
\begin{lemma} The operator $\Phi_{\lambda,\mu}^{(k)}$ satisfies

$i)$ for $v\in V$,
 \[\Phi_{\lambda, \mu}^{(k)} \circ \widetilde \pi_{\lambda+\mu+2k}(t_v) = \big(\widetilde \pi_\lambda(t_v)\otimes \widetilde\pi_\mu(t_v)\big) \circ \Phi_{\lambda, \mu}^{(k)}\]
 
 $ii)$ for $\ell\in L$
 \[\Phi_{\lambda, \mu}^{(k)} \circ\widetilde \pi_{\lambda+\mu+2k}(\ell) = (\widetilde\pi_\lambda(\ell)\otimes \widetilde\pi_\mu(\ell)) \circ \Phi_{\lambda, \mu}^{(k)}\ .
 \]
 \end{lemma}
\begin{proof} For $ii)$
\[\big(\Phi_{\lambda, \mu}^{(k)} \circ \widetilde \pi_{\lambda+\mu+2k} (\ell) h\big) (\xi, \zeta) = \chi(\ell)^{-\frac{\lambda+\mu+2k}{2}+\frac{n}{r}}\,\,\Phi_{\lambda, \mu}^{(k)} (h\circ \ell^*) (\xi, \zeta)
\]
\[= i^{rk}\chi(\ell)^{-\frac{\lambda+\mu+2k}{2}+\frac{n}{r}} (\det \xi)^{\lambda-{\frac{n}{r}}}(\det \zeta)^{\mu-{\frac{n}{r}}}\det (\xi+\zeta)^{\lambda-{\frac{n}{r}}}c_{\lambda-\frac{n}{r}, \mu-\frac{n}{r}}(\xi, \zeta) h(\ell^*\xi+\ell^* \zeta)
\]
On the other hand,
 \[\big(\widetilde \pi_\lambda(\ell)\otimes \widetilde \pi_\mu(\ell)\big) \big(\Phi_{\lambda,\mu}^{(k)} h\big) (\xi,\zeta) = \chi(\ell)^{-\frac{\lambda}{2}+\frac{n}{r}} \chi(\ell)^{-\frac{\mu}{2}+\frac{n}{r}}\big(\Phi_{\lambda+\mu}^{(k)} h\big)(\ell^*\xi,\ell^* \zeta) 
 \]
 \[=i^{rk}\chi(\ell)^{-\frac{\lambda}{2}-\frac{\mu}{2} +\frac{2n}{r}}(\det\ell^*\xi)^{\lambda-\frac{n}{r}}(\det\ell^*\zeta)^{\mu-\frac{n}{r}}(\det\ell^*\xi+\ell^*\zeta)^{-\lambda-\mu-2k+\frac{n}{r}}
 \]
 \[c_{\lambda-\frac{n}{r}, \mu-\frac{n}{r}}(\ell^*\xi,\ell^* \zeta) h(\ell^*\xi+\ell^* \zeta)
 \]
 \[= i^{rk}\chi(\ell)^{-\frac{\lambda}{2}-\frac{\mu}{2}-k+\frac{n}{r}}(\det \xi)^{\lambda-{\frac{n}{r}}}(\det \zeta)^{\mu-{\frac{n}{r}}}\det (\xi+\zeta)^{\lambda-{\frac{n}{r}}}c_{\lambda-\frac{n}{r}, \mu-\frac{n}{r}}(\xi, \zeta) h(\ell^*\xi+\ell^* \zeta)\ .
 \]
 \end{proof}
 It is now possible to transfer these results to the operator ${B_{\lambda, \mu}^{(k)}}^*$ through the Laplace transform.

\begin{proposition} The operator ${B_{\lambda, \mu}^{(k)}}^*$ satisfies
\begin{equation}\label{covtrans}
v\in V,\qquad \qquad{B_{\lambda, \mu}^{(k)}}^*\circ \pi_{\lambda+\mu+2k}(t_v) = \big(\pi_\lambda(t_v)\otimes \pi_\mu(t_v)\big) \circ {B_{\lambda, \mu}^{(k)}}^*
\end{equation}
\begin{equation}\label{covL}
\ell\in L,\qquad \qquad {B_{\lambda, \mu}^{(k)}}^*\circ \pi_{\lambda+\mu+2k}(\ell) =  \big(\pi_\lambda(\ell)\otimes \pi_\mu(\ell)\big) \circ {B_{\lambda, \mu}^{(k)}}^*
\end{equation}
\end{proposition}
\begin{proposition} Let $w\in T_\Omega$. Then for any $(z_1,z_2)\in T_\Omega\times T_\Omega$
\begin{equation}\label{imrepk}
\begin{split}
{B_{\lambda, \mu}^{(k)}}^* k_{\lambda+\mu+2k}^w (z_1,z_2)=\Gamma_\Omega(\lambda+k)\Gamma_\Omega(\mu+k) \\\det(z_1-z_2)^k \det(\frac{z_1-\overline w}{i})^{-\lambda-k} \det(\frac{z_2-\overline w}{i})^{-\mu-k}\ .
\end{split}
\end{equation}
\end{proposition}
\begin{proof} Let $v\in \Omega$. Then $P(v^{\frac{1}{2}} )\in L$, $P(v^{\frac{1}{2}})e=v$ and $\psi_{P(v^{1/2})}= \ln \big(\det v)^{\frac{n}{2r}}\big)$. Hence, by \eqref{kwkgw}
\[k^{iv}_{\lambda+\mu+2k} = (\det v)^{-(\lambda+\mu+2k)} \pi_{\lambda+\mu+2k}\big(P(v^{\frac{1}{2}})\big) k_{\lambda+\mu+2k}^{ie}\]
Use \eqref{covL} to obtain
\[{B_{\lambda, \mu}^{(k)}}^* \big(k_{\lambda+\mu+2k}^{iv}\big)= (\det v)^{-(\lambda+\mu+2k)}\big(\pi_{\lambda}(P(v^{\frac{1}{2}}))\otimes \pi_{\mu}(P(v^{\frac{1}{2}}))\big){B_{\lambda, \mu}^{(k)}}^* k_{\lambda+\mu+2k}^{ie}
\]
which after elementary computations using \eqref{covdet} and \eqref{chiP} yields
\[{B_{\lambda, \mu}^{(k)}}^* \big(k_{\lambda+\mu+2k}^{iv}\big)(z_1,z_2) = \Gamma_\Omega(\lambda+k)\Gamma_\Omega(\mu+k)\dots \]
\[\dots \det (z_1-z_2)^k \det \big(\frac{z_1+iv}{i}\big)^{-\lambda-k}\det\big(\frac{z_2+iv}{i}\big)^{-\mu-k}\ .
\]
Now let $u\in V$. Again by \eqref{kwkgw},  $k_\nu^{u+iv}= \pi_\nu(t_u)k_\nu^{iv}$
and apply again the same trick to compute ${B_{\lambda, \mu}^{(k)}}^* k_\nu^{u+iv}$ and finally obtain \eqref{imrepk}.
\end{proof}
Now comes the test of the covariance property on the coherent states $(k^w_{\lambda+\mu+2k})_{w\in T_\Omega}$.
\begin{proposition}\label{precovB*}
 Let $w\in T_\Omega$. For any $\widetilde g\in \widetilde G$,
\begin{equation}
\big({B_{\lambda, \mu}^{(k)}}^*\circ \pi_{\lambda+\mu+2k}(\widetilde g)\big) k_\nu^w = \big((\pi_\lambda(\widetilde g) \otimes  \pi_\mu(\widetilde g))\circ {B_{\lambda, \mu}^{(k)}}^* \big)k_\nu^w\ .
\end{equation}
\end{proposition}
\begin{proof} For simplicity, set $\nu= \lambda+\mu+2k$
First,
\[\big({B_{\lambda, \mu}^{(k)}}^*\circ \pi_{\nu}(\widetilde g)\big) k_\nu^w(z_1,z_2) = {B_{\lambda, \mu}^{(k)}}^*\big(e^{\frac{r\nu}{2n}}k_\nu^{g(w)}\big)(z_1,z_2)
\]
and this expression is explictly known by \eqref{imrepk}.

On the other hand,
\[\big(\pi_\lambda(g)\otimes \pi_\mu(g)\big) \big({B_{\lambda, \mu}^{(k)}}^*k_\nu^w\big)(z_1,z_2)
\]
\[= \Gamma_\Omega(\lambda+k) \Gamma_\Omega(\mu+k) e^{\frac{r\lambda}{2n}\psi_{g^{-1}}(z_1)} e^{\frac{r\mu}{2n}\psi_{g^{-1}}(z_2)} 
\det\big(g^{-1}(z_1)-g^{-1}(z_2)\big)^k\dots\]\[\dots\det(g^{-1}(z_1)-\overline w)^{-\lambda-k} \det(g^{-1}(z_2)-\overline w)^{-\mu-k}.
\]
It remains to transform this last expression using \eqref{covdet} to conclude. Details are left to the reader.
\end{proof}
\begin{proposition}\label{covB*}
 For any $ g\in \widetilde G$,
\begin{equation}
{B_{\lambda, \mu}^{(k)}}^*\circ \pi_{\lambda+\mu+2k}(g)= \big(\pi_\lambda(g)\otimes \pi_\mu(g)\big) \circ {B_{\lambda, \mu}^{(k)}}^*
\end{equation}
\end{proposition}

\begin{proof} The set of coherent states $\big\{ k_{\lambda+\mu+2k}^w, w\in T_\Omega\big\} $ is total in $\mathcal H_{\lambda+\mu+2k}$ and hence, by the continuity of ${B_{\lambda, \mu}^{(k)}}^*$, Proposition \ref{precovB*} implies Proposition \ref{covB*}.
\end{proof}

\begin{theorem} Let $\lambda, \mu\in \mathbb C$. The bi-differential operator $B_{\lambda,\mu}^{(k)}$ satisfies the following covariance property, valid for any $\ g\in \widetilde G$
\begin{equation}\label{covBeq}
B_{\lambda,\mu}^{(k)} \circ \big(\pi_\lambda(g)\otimes \pi_\mu(g) \big) = \pi_{\lambda+\mu+2k}(g) \circ B_{\lambda,\mu}^{(k)}\ .
\end{equation}
\end{theorem}
\begin{proof} First assume that $\lambda,\mu>1+(r-1)d$. By transposing the result of Proposition \ref{covB*}, the covriance relation is satisfied  for the restriction of $B_{\lambda,\mu}^{(k)}$ to the space $\mathcal H_{\lambda,\mu}$. The general covariance relation, that is to say on $\mathcal O(T_\Omega\times T_\Omega)$ is a consequence of the fact that a holomorphic differential operator $D$ which vanishes on $\mathcal H_{\lambda,\mu}$ vanishes on $\mathcal O(T_\Omega\times T_\Omega)$(apply to the difference of the two sides of \eqref{covBeq}). This is better seen in the \emph{Harish Chandra realization} of $T_\Omega$ as a bounded symmetric domain. Without going into details (see \cite {fk} ch. X), the tube-type domain $T_\Omega$ is holomorphically equivalent to a bounded domain through the \emph{Cayley transform}. There are corresponding weighted Bergman spaces, and an equivalent realization of the holomorphic discrete series. Now in this model, all holomorphic polynomials belong to the weighted Bergman spaces, and hence if a holomorphic differential operator vanishes on the weighted Bergman space, it vanishes on all polynomials and hence is null. 

To finish the proof, let $\lambda, \nu$ be arbitrary in $\mathbb C$. As $B_{\lambda,\mu}^{(k)}$ and the representations  depend holomorphically on the parameters $\lambda, \mu$, the general result follows by a standard argument.  
\end{proof}

\section{The family $\big(C^{(k)}_{\lambda,\mu}\big)_{k\in \mathbb N}$}

\subsection{The orthogonality property}
\begin{proposition} The polynomials $C_{\lambda,\mu}^{(k)}$ form an orthogonal family in \break $L^2\left(\rrbracket -e,+e\,\llbracket, \det(e-v)^{\lambda}\det(e+v)^{\mu}\right)$.
\end{proposition}

\begin{proof}
 For $k\neq l\in \mathbb N$, the subspaces  $Im \left({B^{(k)}_{\lambda,\mu}}^*\right)$ and $Im \left({B^{(l)}_{\lambda,\mu}}^*\right)$ are two non equivalent (hence orthogonal) components in the decomposition of the tensor product $\mathcal H_{\lambda,\mu}$ (see e.g. \cite{pz}). In particular, ${B^{(k)}_{\lambda, \mu}}^{\!*} k_{\lambda+\mu+2k}^{ie}$ and ${B^{(l)}_{\lambda, \mu}}^{\!\!*}k_{\lambda+\mu+2l}^{ie}$ are orthogonal in $\mathcal H_{\lambda,\mu}$. Under inverse Laplace transform, this amounts to
 \[0= \int_\Omega \int_\Omega \Phi_{\lambda,\mu}^{(k)} (\psi_{\lambda+\mu+2k})(\xi,\zeta)\Phi_{\lambda,\mu}^{(l)} (\psi_{\lambda+\mu+2l})(\xi,\zeta)(\det \xi)^{-\lambda+\frac{n}{r} }(\det \zeta)^{-\mu+\frac{n}{r}} d\xi d\zeta
 \]
After some simplification the integral becomes
\[ \int_\Omega\int_\Omega (\det \xi)^{\lambda-\frac{n}{r}} (\det \zeta)^{\mu-\frac{n}{r}} e^{-2\tr (\xi+\zeta)}c_{\lambda-\frac{n}{r},\,\mu-\frac{n}{r}}^{(k)}(\xi,\zeta) c_{\lambda-\frac{n}{r},\,\mu-\frac{n}{r}}^{(l)}(\xi,\zeta) d\xi d\zeta\ .
\]
Use the change of coordinates $(\xi,\zeta) = \boldsymbol \iota(\eta,v)$ to write the integral as
\[\int_\Omega \int_{\rrbracket -e,+e\,\llbracket}\!\!\!\!\!\! \!\!\!\!\!\!e^{-2\tr \eta} (\det \eta)^{k+l+\frac{n}{r}}\det(e-v)^{\lambda-\frac{n}{r}}\det(e+v)^{\mu-\frac{n}{r}}\ \dots\]\[\dots \ C_{\lambda-\frac{n}{r},\,\mu-\frac{n}{r}}^{(k)}(v) C_{\lambda-\frac{n}{r},\,\mu-\frac{n}{r}}^{(l)}(v) d\eta\, dv\ .
\]
As 
\[\int_\Omega e^{-2\tr \eta} (\det \eta)^{k+l+\frac{n}{r}} d\eta >0
\]
follows
\[\int_{\rrbracket -e,+e\,\llbracket} C_{\lambda-\frac{n}{r},\mu-\frac{n}{r}}^{(k)}(v) \,C_{\lambda-\frac{n}{r},\mu-\frac{n}{r}}^{(l)}(v)\det(e-v)^{\lambda-\frac{n}{r}}\,\det(e+v)^{\mu-\frac{n}{r}} \, dv = 0\ .
\]
Up to the change of $\lambda-\frac{n}{r}, \mu-\frac{n}{r}$ to $\lambda,\mu$, this proves the statement.
\end{proof}

\subsection{Further investigations}

The family of polynomials $C_{\lambda,\mu}^{(k)}$ reminds of the classical Jacobi polynomials, appearing in the simplest  example $V=\mathbb R$ (cf \cite{kp} or \cite{c}). They have in common at least three properties (see \cite{c} for details) :
\smallskip

$i)$  they 	are defined through a \emph{Rodrigues formula}
\smallskip

$ii)$ they satisfy a recurrence relation relating $C_{\lambda,\mu}^{(k)}$ and $C_{\lambda+1,\mu+1}^{(k-1)}$
\smallskip

$iii)$ for appropriate values of the parameters $\lambda,\mu$, orthogonality properties with respect to a certain measure $d\nu_{\lambda, \mu}$  on an "interval". 
\smallskip

Notice however that, strictly speaking, the Rodrigues formula and the recurrence relation are expressed in terms of the polynomials $c_{\lambda,\mu}^{(k)}$, but in principle they could be converted in similar formulas for the $C_{\lambda,\mu}^{(k)}$.

For the appropriate values of the parameters mentioned in $iii)$, the Jacobi polynomials are not only orthogonal but form an \emph{orthogonal basis} of the Hilbert space of square-integrable functions for the measure $d\nu_{\lambda, \mu}= (1-x)^\lambda(1+x)^\mu dx$. This is lacking in our more general case. The polynomials $C_{\lambda, \mu}^{(k)}$ are  invariant by the automorphisms group $Aut(V)$ of the Jordan algebra $V$.  From the spectral theorem for Euclidean Jordan algebras, it is easy to see that such an invariant polynomial (or more generally invariant function) depends on $r$ variables. So one can guess that, in order to be a basis of the Hilbert pace of invariant functions square-integrable w.r.t. $d\nu_{\lambda, \mu}= \det(e-v)^\lambda \det(e+v)^\mu dv$, a family of polynomials should depend on a $r$-tuple of integers. The decomposition of the tensor product of two scalar holomorphic representations contains, beyond the scalar holomorphic representations, many non scalar components (see \cite{pz}), indexed by a $r$-tuple of integers, which are not considered in the present article. To these components correspond more polynomials, which presumably could enlarge the family of the orthogonal polynomials $C_{\lambda,\mu}^{(k)}, k\in \mathbb N$, in order to obtain an orthogonal \emph{basis} of the Hilbert space of square-integrable functions with respect to the measure $d\nu_{\lambda,\mu}$.

Let me finish by addressing another question. The classical Jacobi polynomials have expressions in terms of certain hypergeometric functions. A theory (not that much developed) of generalized hypergeometric functions exists for Euclidean Jordan algebra, see \cite{fk} ch. XV. A complementary investigation  would be to try to relate the polynomials $C_{\lambda, \mu}^{(k)}$ with these generalized hypergeometric functions.

\medskip
\footnotesize{\noindent Address\\ Jean-Louis Clerc, Universit\'e de Lorraine, CNRS, IECL, F-54000 Nancy, France
\medskip

\noindent \texttt{{jean-louis.clerc@univ-lorraine.fr
}}

\end{document}